\documentclass[12pt, reqno]{amsart}
\usepackage{amsmath,amsthm,amscd,amsfonts,amssymb,graphicx,subfigure,color}
\usepackage[bookmarksnumbered,colorlinks,plainpages]{hyperref}
\usepackage{ragged2e}
\usepackage[pagewise]{lineno}
\hypersetup{colorlinks=true,linkcolor=black, anchorcolor=black,
	citecolor=black, urlcolor=black, filecolor=magenta, pdftoolbar=true}
\newtheorem{theorem}{Theorem}
\newtheorem{lemma}{Lemma}

\newtheorem{corollary}{Corollary}
\theoremstyle{definition}
\newtheorem{definition}{Definition}
\newtheorem{example}{Example}

\theoremstyle{remark}
\newtheorem{remark}{Remark}
\numberwithin{equation}{section}

\begin{document}
 
\begin{center}
{\textbf{\Large Quasi Strongly $E$-preinvexity and its Relationships with Nonlinear Programming}}\\
\bigskip

{\textbf {Akhlad Iqbal and Askar Hussain}}\\ 					
Department of  Mathematics, 
\\ Aligarh Muslim University, Aligarh-202002, India\\
Email: askarhussain59@gmail.com and akhlad6star@gmail.com

\end{center}

\bigskip
\noindent 
\textbf{Abstract}:  In this paper, we extend the class of strongly $E$-preinvex and strongly  $E$-invex   functions to quasi strongly $E$-preinvex, quasi strongly $E$-invex  and pseudo strongly $E$-invex functions. Some nontrivial suitable examples have been constructed in support of our definitions. Several interesting properties and relationships of these functions are discussed. Furthermore, to show the application of our results, we consider a nonlinear programming problem and show that the local minimum point is also a strictly global minimum.\\


\section{Introduction} 
Convexity is an important branch of mathematics, which has many applications in pure and applied mathematics such as optimization theory,  engineering science and economics. Because of its vast applications, Youness \cite{Youness1} introduced  $E$-convexity and later discussed the concept of semi strongly  $E$-convex functions, see \cite{Youness3, Youness4}. Fulga et al. \cite{Fulga} generalized this concept and defined  $E$-invexity and  $E$-preinvexity. Later, Jaiswal et al. \cite{Jaiswal} introduced generalized $E$-convex function and differentiable $E$-invexity.  They \cite{Jaiswal} discussed several interesting properties and derived the existence of optimal solution from the set of feasible solutions for a $NLPP$. Furthermore, the extension of strongly $E$-convexity to geodesic strongly $E$-convexity from linear space to Riemannian manifolds was developed by Kilickman et al. \cite{Adem1}.  They \cite{Adem1} have attempted  to preserve several interesting properties and results of linear space for the Riemannian manifolds. Afterwards, Iqbal et al. \cite{Iqbal1} defined geodesic $E$-prequasi-invex functions and discussed its characteristics with $NLPP$.   Hussain et al. \cite{Hussain1} introduced quasi strongly  $E$-convex functions and later, Iqbal et al. \cite{Hussain2} extended the concept of strongly  $E$-convexity to strongly  $E$-invexity and strongly  $E$-preinvexity.\\

Motivated by research works on convexity \cite{ Abou, Bazaraa, Ben,Adem2,Jeyakumar,Pini,Mond}, we introduce a new class of generalized strongly  $E$-preinvex   functions, named as quasi strongly  $E$-preinvex   function. The paper is divided as follows$\colon$ In Section \ref{Sec:Pre}, we recollect some basic definitions and results. Section \ref{Sec:results} is devoted to deriving several interesting properties and to discuss our main results. We also construct several examples in this section. In Section \ref{Sec:Non-linear programming problem}, we consider the non-linear programming problem for the quasi strongly $E$-preinvex functions and show that the local minimum point is a strictly minimum point.  The conclusion and motivation of this paper is explained in the last section.

\section{Preliminaries}\label{Sec:Pre}

Let $ {R}^{n}$  represent Euclidean space of dimension $n$. The  maps $E\colon {R}^{n}\rightarrow  {R}^{n}$ and $\Psi\colon R^{n}\times R^{n}\to R^{n}$ have been considered throughout the paper. Fulga et al. \cite{Fulga} defined  $E$-preinvexity as follows$\colon$

\begin{definition} \cite{Fulga}\label{dfn1}
	A set $S\subseteq {R}^{n}$ is called  $E$-invex with respect to (w.r.t.) $\Psi$, if for all $s,t\in S$ and $\lambda\in[0,1]$, we have  
\begin{eqnarray*}
	Et+\lambda\Psi(Es,Et)\in S.
\end{eqnarray*}
\end{definition}
\begin{definition}\cite{Fulga}
	Let $S\subseteq  {R}^{n}$ be an $E$-invex set. A function $h\colon S\rightarrow  {R}$ is  called  $E$-preinvex w.r.t. $\Psi$ on $S$, if for all $s,t\in S$ and $\lambda\in[0,1]$, we have  
\begin{eqnarray*}
	 h(Et+\lambda\Psi(Es,Et))\leq \lambda h(Es)+(1-\lambda)h(Et).
\end{eqnarray*}
\end{definition}

\begin{definition}\cite{Fulga}
	Let $S\subseteq  {R}^{n}$ be an  $E$-invex set. A function $h\colon S\rightarrow  {R}$ is  called  $E$-prequasi-invex w.r.t. $\Psi$  on $S$, if for all $s,t\in S$ and $\lambda\in[0,1]$, we have   
\begin{eqnarray*}
	 h(Et+\lambda\Psi(Es,Et))\leq \max\{h(Es),h(Et)\}.
\end{eqnarray*}
\end{definition}

Iqbal et al. \cite{Hussain2} introduced strongly  $E$-invexity and discussed some results as follows$\colon$

\begin{definition} \cite{Hussain2}\label{definition}
	A set $S\subseteq  {R}^{n}$ is  called strongly  $E$-invex (SEI) w.r.t. $\Psi\colon {R}^{n}\times  {R}^{n}\rightarrow  {R}^{n}$, if  for all $s,t\in S$, $\lambda\in[0,1]$ and $\alpha\in[0,1]$, we have    
\begin{eqnarray*}
	\alpha t+Et+\lambda\Psi(\alpha s+Es,\alpha t+Et)\in S.
\end{eqnarray*}
\end{definition}

\begin{remark}
	If ${\alpha}=0$, then ${S}$ reduces to ${E}$-invex set, which is the Definition \ref{dfn1}.	If $\Psi(\alpha s+Es, \alpha t+Et)=(\alpha s+Es-\alpha t-Et)$, $\forall s,t\in  {R}^{n}$, then we regain the  strongly  $E$-convexity defined by
	Youness \cite{Youness3}.
\end{remark}
\begin{lemma}\cite{Hussain2}\label{lem:1} 
	Let ${S}\subseteq  { {R}^{n}}$ be a SEI set  w.r.t. $ {\Psi}$, then $ {E}( {S})\subseteq {S}$.
\end{lemma}
\begin{lemma}\cite{Hussain2}
	Let  $\{ {S}_{i}\}_{i\in I}$ be a the family of SEI sets w.r.t. $ {\Psi}$, and $ {S}_{i}\subseteq { {R}^{n}},$ for every $i\in I$, then the intersection $\bigcap\limits_{i\in I} {S}_{i}$ is strongly $ {E}$-invex set  w.r.t. $ {\Psi}$. 
\end{lemma}

\begin{definition} \cite{Hussain2}
	Let $S\subseteq  {R}^{n}$ be a SEI set. A function $h\colon S\rightarrow  {R}$ is  called strongly $E$-preinvex (SEP) w.r.t. $\Psi$  on $S$, if 
\begin{eqnarray*}
	h(\alpha t+Et+\lambda\Psi(\alpha s+Es,\alpha t+Et))\leq \lambda h(Es)+(1-\lambda)h(Et),
\end{eqnarray*}
$\forall s,t\in S,\alpha\in[0,1]~~\&~~\lambda\in[0,1]$.
\end{definition}

\begin{definition}\cite{Hussain2}\label{5P} 
	Let ${S}\subseteq  {R}^{n}$ be an open SEI set and  ${h}\colon {S}\rightarrow {R}$ be a differentiable function on ${S}$. Then, ${h}$ is  called  SEI w.r.t. $ {\Psi}$ on $ {S}$, if
	\begin{eqnarray*}
		\nabla  {h}( {E}t)\ {\Psi}( {\alpha} s+ {E}s, {\alpha} t+ {E}t)^{T}\leq  {h}( {E}s)- {h}( {E}t),
	\end{eqnarray*}
  $\forall s,t\in  {S}$ and $ {\alpha}\in {[0,1]} $.
\end{definition} 
 
Iqbal et al. \cite{Hussain2} introduced the following $\mathbf{Condition~A}$ to show that a differentiable SEI function with this condition is SEP function.
 
\noindent
  For an onto map  ${E}:{S}\rightarrow  {S}$, let $ {M}\subseteq  {R}^{n}$ be a SEI set  w.r.t. $ {\Psi}$, assume for each $s,t\in {S},~ {\alpha}\in {[0,1]},{\lambda}\in {[0,1]},~{\exists}~\bar{v}\in {S}$ s.t.  
$${E}\bar{v}= {\alpha} t+{E}t+ {\lambda} {\Psi}\left({\alpha} s+ {E}s, {\alpha} t+{E}t\right)\in{S}.$$
Then, ${\Psi}$ satisfies $\mathbf{Condition~A}$ if:

\begin{equation*}
	{\bf A_{1}:~~}~~ {\Psi}\big({\alpha} t+ {E}t, {\alpha} \bar{t}+ {E}\bar{t}\big)=- {\lambda}\big( {\alpha}\bar{t}+ {\Psi}({\alpha} s+ {E}s, {\alpha} t+{E}t)\big), 
\end{equation*}
\begin{equation*}
	\hspace{.8cm} {\bf A_{2}:~~}~~{\Psi}\left({\alpha}s+ {E}s,{\alpha} \bar{t}+{E}\bar{t}\right)=(1-{\lambda})\left({\alpha}\bar{t}+{\Psi}({\alpha}s+{E}s,{\alpha}t+{E}t)\right). 
\end{equation*} 
\noindent
 For ${\alpha}=0$ and ${E}s=s~\forall s\in S$, the $\mathbf{Condition~A}$ reduces to the $\mathbf{Condition~C}$ defined by Mohan et al. \cite{Mohan}.

\begin{theorem}\cite{Hussain2} \label{thm:1}
	Let ${S}\subseteq  {R}^{n}$ be an open SEI set  w.r.t. ${\Psi}$ and an onto map $ {E}\colon {S}\rightarrow  {S}$. If the function $ {h}\colon{S}\rightarrow  {R}$ is differentiable  SEI w.r.t. ${\Psi}$ on ${S}$ and $ {\Psi}$ satisfies Condition $ {A}$. Then, $ {h}$ is   SEP w.r.t. $ {\Psi}$ on $ {S}$.
\end{theorem} 

\section{{Mains Results}}\label{Sec:results}

In this section, we introduce the quasi strongly $E$-preinvex (QSEP) function and discuss several interesting properties of this function.  
\begin{definition}
	Let $S\subseteq  {R}^{n}$ be a nonempty SEI set. A  function $h\colon {R}^{n}\rightarrow {R}$ is   called QSEP w.r.t. $\Psi$ on   $S$, if 
	\begin{eqnarray*}
		h(\alpha t+Et+\lambda\Psi(\alpha s+Es,\alpha t+Et))\leq \max\{h(Es),h(Et)\},
	\end{eqnarray*}
	$\forall s,t\in S,\alpha\in[0,1]~~\&~~\lambda \in[0,1]$.\\
	
	\noindent  
	The function $h$ is  called a strictly QSEP function, if the inequality is strict and $h(Es)\neq h(Et)$, $\forall s,t\in S, \alpha\in[0,1]~~\&~~\lambda\in(0,1)$.
\end{definition}
\begin{remark}
	QSEP function w.r.t.  $\Psi$ is an $E$-prequasi-invex defined by Fulga \cite{Fulga}, if $\alpha=0$.
\end{remark}
\begin{example}
	Suppose $h\colon {R}\rightarrow  {R}$ be a defined as$\colon$
	\begin{equation*}
		h(s)=
		\begin{cases}
			1 & \text{if } s>0, \\
			-s & \text{if } s\leq0. 
		\end{cases}
	\end{equation*} 
	
	\noindent
	and $E\colon {R}\rightarrow  {R}$ be a map defined as 	$Es=|s|$ and $\Psi\colon {R}\times  {R}\rightarrow  {R}$ be defined as$\colon$
	\begin{equation*}
		\Psi(s,t)=
		\begin{cases}
			-t & \text{if } s\neq t, \\
			0 & \text{if } s=t. 
		\end{cases}\hspace{.5cm}
	\end{equation*}
To show this $\forall s,t\in R, \lambda\in[0,1]$ and $\alpha\in[0,1]$, three possible cases are given:

\noindent
$\mathbf{Case (1)\colon}$ $s\neq t>0$ or $s\neq t<0$ we have
\begin{eqnarray*}
	h(\alpha t+Et-\lambda (\alpha t+ Et))&\leq&\max\{h(Es),h(Et)\}\\
	h((1-\lambda)(1+\alpha)t)&\leq&\max\{h(s),h(t)\}\\ ~~or~~
	h((1-\lambda)(\alpha t+ Et))&\leq&\max\{h(Es),h(Et)\}\\
	h((1-\lambda)(1-\alpha)t)&\leq&\max\{h(-s),h(-t)\}.
\end{eqnarray*}	
	$\mathbf{Case (2)\colon}$ $s=t>0$ or $s=t<0$ we have
	\begin{eqnarray*}
		\hspace{2.4cm}h(\alpha t+Et)&\leq&\max\{h(Et),h(Et)\}\\
			h((1+\alpha)t)&\leq&\max \{h(t),h(t)\}\\
				h((1+\alpha)t)&\leq& h(t)\\
			~or~h(\alpha t+Et)&\leq&\max\{h(Et),h(Et)\}\\
			h((1-\alpha)t)&\leq& \max\{h(-t),h(-t)\}\\
			h((1-\alpha)t)&\leq& h(-t).
	\end{eqnarray*}
	$\mathbf{Case (3)\colon}$ $s>0, t=0$,or $s<0,t=0$, we have 
	\begin{eqnarray*}
		h(0)&\leq&\max\{h(Es),h(E0)\}\\
		0&\leq&\max\{h(s),h(0)\}\\ 
			0&\leq&\max\{1,0\}.
	\end{eqnarray*} 
Therefore, in all cases, we have
\begin{eqnarray*}
	\begin{split}
		h(\alpha t+Et+\lambda\Psi(\alpha s+Es,\alpha t+Et))\leq \max\{h(Es),h(Et)\}. 
	\end{split}
\end{eqnarray*} 
	Hence, the function $h$ is QSEP as well as $E$-prequasi-invex  w.r.t. $\Psi$ but it is not strongly $E$-preinvex. In particular, at points $s=0,t=1,\alpha=\frac{1}{2}$ and $\lambda=\frac{1}{2}$, we get
\begin{eqnarray*}
	h\left(\alpha 1+E1+\lambda \Psi(\alpha 0+E0,\alpha 1+E1)\right)&=&h\left(\frac{3}{2}+\frac{1}{2} \Psi\left(0,\frac{3}{2}\right)\right)\\&=&h\left(\frac{3}{4}\right)\\&=&1.
\end{eqnarray*}
	However,
\begin{eqnarray*}
	 \hspace{1.5cm}\lambda h(E0)+(1-\lambda)h(E1)&=&\frac{1}{2} h(0)+\frac{1}{2}h(1)\\&=&\frac{1}{2},
\end{eqnarray*} 
which implies 
\begin{eqnarray*}
	h(\alpha t+Et+\lambda\Psi(\alpha s+Es,\alpha t+Et))\nleq \lambda h(Es)+(1-\lambda)h(Et).
\end{eqnarray*}
	Hence, $h$ is not SEP function.\\

	An $E$-prequasi-invex function w.r.t. $\Psi$  need not be necessarily  QSEP function as shown in the following Example$\colon$
\end{example}

\begin{example}
	
	Suppose $h\colon {R}\rightarrow  {R}$ be a defined as 
	\begin{equation*}
		\hspace{.3cm} h(s)=
		\begin{cases}
			1 & \text{if } s>0 \\
			-s & \text{if } s\leq0, 
		\end{cases}
	\end{equation*}
$E\colon {R}\rightarrow {R}$ is defined as $Es=-s^{2}~\forall s\in{R}$,  
	and $\Psi\colon R\times R\to R$ is defined as
	\begin{equation*}
		\Psi(s,t)=
		\begin{cases}
			-t & \text{if } s\neq t, \\
			0 & \text{if } s=t. 
		\end{cases}.
	\end{equation*}
 To show this $\forall s,t\in R, \lambda\in[0,1]$ and $\alpha\in[0,1]$, we have three possible cases are given:
 
\noindent
	$\mathbf{Case (1)\colon}$ $s\neq t>0$ or $s\neq t<0$ we have
	\begin{eqnarray*}
		 \hspace{.2cm}h(Et-\lambda Et)&\leq&\max\{h(Es),h(Et)\}\\
		  h((1-\lambda) Et)&\leq&\max\{h(-s^{2}),h(-t^{2})\}\\
		  (1-\lambda)t^{2}&\leq& t^{2}.
	\end{eqnarray*}	  
$\mathbf{Case (2)\colon}$ $s=t>0$ or $s=t<0$ we have
\begin{eqnarray*}
\hspace{2cm} h(Et)&\leq&\max\{h(Et),h(Et)\}\\
	h(-t^{2})&\leq&\max\{h(-t^{2}),h(-t^{2})\}\\
	 t^{2}&\leq& t^{2}.
\end{eqnarray*}
$\mathbf{Case (3)\colon}$ $s>0, t=0$ or $s<0,t=0$, we have 
\begin{eqnarray*}
	h(0)&\leq&\max\{h(Es),h(E0)\}\\
	h(0)&\leq&\max\{h(-s^{2}),h(0)\}\\
	0&\leq& s^{2}.
\end{eqnarray*}
 
Therefore, in all cases, we have
\begin{eqnarray*}
	h(Et+\lambda\Psi(Es, Et))\leq\max\{h(Es),h(Et)\}, \forall s,t\in R, \lambda\in[0,1].
\end{eqnarray*} 
	Hence, the function $h$ is $E$-prequasi-invex  but it is not QSEP function. At points $s=0,t=-1,\alpha=1$ and $\lambda=0$, we get
\begin{eqnarray*}
	h(\alpha (-1)+E(-1)+\lambda \Psi(\alpha (0)+E(0),\alpha (-1)+E(-1))&=&h(-1-1)\\&=&h(-2)\\&=&2.
\end{eqnarray*}
	However,
	\begin{eqnarray*}
		\hspace{6.4cm}\max\{h(0),h(-1)\}&=&\max\{0,1\}\\&=&1,
	\end{eqnarray*}which implies
\begin{eqnarray*}
	h(\alpha t+Et+\lambda\Psi(\alpha s+Es,\alpha t+Et))\nleq \max\{h(Es),h(Et)\}.
\end{eqnarray*}	 
\end{example}

\begin{theorem}\label{thm:2}
	Let $S\subseteq  {R}^{n}$ be a SEI set. If $h\colon {R}^{n}\rightarrow  {R}$ is QSEP w.r.t. $\Psi$ on   $S$, then $ h(\alpha t+Et)\leq h(Et)$, for all $t\in S$ and $\alpha \in[0,1]$.  
\end{theorem}
\begin{proof}
	Since the function $h$ is QSEP  on a SEI set $S$,   $\forall s,t\in S,\alpha \in[0,1]~~\&~~ \lambda\in[0,1]$, we get
	\begin{equation*}
		\alpha t+Et+\lambda\Psi(\alpha s+Es,\alpha t+Et)\in S
	\end{equation*} and
	\begin{equation*}
		h(\alpha t+Et+\lambda\Psi(\alpha s+Es,\alpha t+Et))\leq \max\{h(Es),h(Et)\}.
	\end{equation*}
	Thus, for $s=t$ and take $\lambda=0$, we get $ h(\alpha t+Et)\leq h(Et), $ for all $t\in S$.   
\end{proof}
 
\begin{theorem} \label{thm:3}
	Suppose $S\subseteq  {R}^{n}$ be a SEI set. If the functions $h_{j}\colon {R}^{n}\rightarrow  {R}$, $1\leq j\leq n $, are  non negative  QSEP w.r.t. $\Psi$ on  $S$, then the linear combination 
	\begin{eqnarray*}
		h=\sum\limits_{j=1}^{n}a_{j} h_{j},
	\end{eqnarray*} 
for $a_{j}\geq0,~~~ 1\leq j\leq n$,  is QSEP function on $S$.  
\end{theorem}

\begin{proof}
	Since $h_{j}, ~~1\leq j\leq n$, are QSEP functions on SEI set $S$, then for every $s,t\in S$, $\alpha\in[0,1]$ and $\lambda\in[0,1]$, we get\\
	\begin{equation*}
		\alpha t+Et+\lambda\Psi(\alpha s+Es,\alpha t+Et)\in S
	\end{equation*} and
	
	\begin{eqnarray*}
		h\left(\alpha t+Et+\lambda\Psi(\alpha s+Es,\alpha t+Et)\right)&=&\sum\limits_{j=1}^{n}a_{i} h_{i}\left(\alpha t+Et+\lambda\Psi(\alpha s+Es,\alpha t+Et)\right)\\
		&\leq & \max\left\{\sum\limits_{j=1}^{n}a_{i} h_{j}(Es),  \sum\limits_{j=1}^{n}a_{j} h_{j}(Et)\right\}\\
		&=& \max\left\{h(Es),h(Et)\right\}.
	\end{eqnarray*}
	Hence, $h(s)$ is QSEP  on $S$.
\end{proof}
 
\begin{theorem}  \label{thm:4} 
	Let $S\subseteq {R}^{n}$ be a SEI and $\{h_{i}\}_{i\in I}$ be a collection of functions defined on $S$ s.t. $\sup\limits_{i\in I} h_{i}(s)$ exists in $ {R}$, $\forall s\in S$. Let $h\colon S\rightarrow R$ be a function defined by $h(s)=\sup\limits_{i\in I}h_{i}(s),\forall s\in S$. If the functions $h_{i}\colon S\rightarrow  {R}$ , for every $i\in I$, are all QSEP w.r.t. $\Psi$ on   $S$, then $h$ is QSEP on $S$.   
\end{theorem}
\begin{proof}
	Since functions $h_{i}\colon S\rightarrow  {R}$, for every $i\in I$, are QSEP on   $S$, then $\forall s,t\in S$, $\alpha \in [0,1]~~\&~~\lambda \in[0,1]$, we get
	\begin{eqnarray*}
		h_{i}(\alpha t+Et+\lambda\Psi(\alpha s+Es,\alpha t+Et))&\leq& \max\left\{h_{i}(Es),h_{i}(Et)\right\},\\
		\sup\limits_{i\in I}h_{i}(\alpha t+Et+\lambda\Psi(\alpha s+Es,\alpha t+Et))&\leq& \max\left\{\sup\limits_{i\in I} h_{i}(Es),\sup\limits_{i\in I}h_{i}(Et)\right\},\\
		&=&\max\left\{h(Es),h(Et)\right\},\\
		h(\alpha t+Et+\lambda\Psi(\alpha s+Es,\alpha t+Et))&\leq& \max\left\{h(Es),h(Et)\right\}.
	\end{eqnarray*} 
	Therefore, $h$ is QSEP on $S$.
\end{proof}
\begin{theorem} \label{thm:5}
	Let $S\subseteq  {R}^{n}$ be a SEI set. Let the function $h\colon {R}^{n}\rightarrow  {R}$ be QSEP w.r.t. $\Psi$ on $S$. Let $g\colon {R}\rightarrow  {R}$ be a positively homogeneous non-decreasing function, then $g\circ h$ is QSEP on $S$.
\end{theorem}
\begin{proof}
	Since the function $h$ is QSEP  on   $S$, $\forall s,t\in S$, $\alpha\in[0,1]~~\&~~\lambda\in[0,1]$, we get
\begin{eqnarray*}
	\alpha t+Et+\lambda\Psi(\alpha s+Es,\alpha t+Et)\in S,
\end{eqnarray*}
	and
\begin{eqnarray*}
	h(\alpha t+Et+\lambda\Psi(\alpha s+Es,\alpha t+Et))\leq \max\left\{h(Es),h(Et)\right\}.
\end{eqnarray*}
	Since $g$ is a  positively homogeneous non-decreasing, we get
	\begin{eqnarray*}
		\left(g\circ h\right)\left(\alpha t+Et+\lambda\Psi(\alpha s+Es,\alpha t+Et)\right)&\leq&g\circ\left(\max\left\{h(Es),h(Et)\right\}\right),\\
		&=& \max\left\{\left(g\circ h\right)(Es),\left(g\circ h\right)(Et)\right\}.
	\end{eqnarray*}
	Hence, the function $g\circ h$ is QSEP on $S$.
\end{proof}
 
In the following, we show a relation between SEP functions and QSEP functions.

\begin{theorem}\label{thm:6}
	Let  $S\subseteq  {R}^{n}$ be a SEI set. If a function $h\colon {R}^{n}\rightarrow  {R}$ is SEP w.r.t. $\Psi$ on   $S$ and $h(Es)\leq h(Et)$, $\forall s,t\in S$. Then, $h$ is QSEP on $S$. 
\end{theorem}
\begin{proof}
	Let $h$ be a SEP function on $S$, then $\forall s,t\in S,\alpha\in[0,1]~~\&~~\lambda\in[0,1]$, we get 
\begin{eqnarray*}
	h(\alpha t+Et+\lambda\Psi(\alpha s+Es,\alpha t+Et))\leq \lambda h(Es)+(1-\lambda)h(Et).
\end{eqnarray*}
	Since $h(Es)\leq h(Et)$, for each $s,t\in S$.
	Then, \begin{eqnarray*}
		h(\alpha t+Et+\lambda\Psi(\alpha s+Es,\alpha t+Et))\leq h(Et)=\max\{h(Es),h(Et)\}.
	\end{eqnarray*}
	Thus, $h$ is QSEP on $S$.  
\end{proof}
Motivated by Barani  \cite{Poury1} and Azagra \cite{Azagra},
in the following theorem, we establish a relationship between quasi strongly $E\times E$-preinvex functions and QSEP, which is the generalization of Proposition 3.2 in \cite{Poury1}.
 \begin{theorem}
	Let $S\subseteq R^{n}$ be a SEI set w.r.t. $\Psi$ and $F\colon S\times S\rightarrow R$ be a continuous quasi strongly E$\times$ $E$-preinvex function w.r.t. $\Psi\times \Psi$, $i.e.,$ $F$ is quasi strongly E$\times$ $E$-preinvex function to each variable. Then, $g\colon S\rightarrow R$ defined by 
	\begin{eqnarray*}
		g(s)=\inf\limits_{t\in S}F(s,t),
	\end{eqnarray*} 
is QSEP function w.r.t. $\Psi$.
\end{theorem}
\begin{proof}
	Let $s_{0},s_{1}\in S$ be given and $\epsilon>0$ be an arbitrary. Since $S$ is SEI set w.r.t. $\Psi$.   
	\begin{eqnarray*}
		\alpha s_{1}+Es_{1}+\lambda\Psi(\alpha s_{0}+Es_{0},\alpha s_{1}+Es_{1})\in S,~\forall~\alpha\in[0,1],~ \lambda\in[0,1].
	\end{eqnarray*}
	By the definition of infimum, there exist $t_{0},t_{1}\in S$ s.t.
	\begin{eqnarray*}
		F(Es_{1},Et_{1})<g(Es_{1})+\epsilon,~ 	F(Es_{0},Et_{0})<g(s_{0})+\epsilon.
	\end{eqnarray*}
	By strongly  $E$-invexity of $S$ w.r.t. $\Psi$, we get  
	\begin{eqnarray*}
		\alpha t_{1}+Et_{1}+\lambda\Psi(\alpha t_{0}+Et_{0},\alpha t_{1}+Et_{1})\in S,~\forall~\alpha\in[0,1],~ \lambda\in[0,1].
	\end{eqnarray*}
	It follows that the set $S\times S$ is strongly $E\times E $-invex set w.r.t. $\Psi\times \Psi$.\\ $i.e.$, for every $(s_{0},t_{0}),~(s_{1},t_{1})\in S\times S, \alpha\in[0,1],\lambda\in[0,1]$, we have
	 
	\begin{multline*}
		\alpha (s_{1},t_{1})+(E\times E)(s_{1},t_{1})+\lambda(\Psi\times\Psi)\big(\alpha (s_{0},t_{0})+(E\times E)(s_{0},t_{0}),\\(\alpha (s_{1},t_{1})+(E\times E)(s_{1},t_{1}))\big)
			\end{multline*}
		\begin{multline*}
		=(\alpha (s_{1},t_{1})+(E\times E)(s_{1},t_{1})+\lambda(\Psi\times\Psi)\big((\alpha s_{0}+Es_{0},\alpha s_{1}+Es_{1}),\\(\alpha t_{0}+Et_{0},\alpha t_{1}+Et_{1})\big)
			\end{multline*}
		\begin{multline*}
		=(\alpha s_{1},\alpha t_{1})+(Es_{1},Et_{1})+\lambda\{\Psi(\alpha s_{0} +Es_{0},\alpha s_{1}+Es_{1}),\Psi(\alpha t_{0}+Et_{0},\alpha t_{1}+Et_{1})\}
			\end{multline*}
		\begin{multline*}
		=\big(\alpha s_{1}+Es_{1}+\lambda\Psi(\alpha s_{0}+Es_{0},\alpha s_{1}+Es_{1}),\\\alpha t_{1}+Et_{1}+\lambda\Psi(\alpha t_{0}+Et_{0},\alpha t_{1}+Et_{1})\big)\in S\times S, 
	\end{multline*}  
	where $E\times E\colon R^{n}\times R^{n}\rightarrow R^{n}\times R^{n}$ and $\Psi\times \Psi\colon (R^{n}\times R^{n})\times( R^{n}\times R^{n})\rightarrow R^{n}\times R^{n}$ are the maps.\\
	By the Definition of infimum and the quasi strongly $E\times E$-preinvexity of $F$ w.r.t. $\Psi\times \Psi,$\\  we have\\
	$g\left(\alpha s_{1}+Es_{1}+\lambda\Psi(\alpha s_{0}+Es_{0},\alpha s_{1}+Es_{1})\right)$
	\begin{eqnarray*}
		\hspace{1.5cm}=\inf\limits_{t\in S} F(\alpha s_{1}+Es_{1}+\lambda\Psi(\alpha s_{0}+Es_{0},\alpha s_{1}+Es_{1}),t)
		\end{eqnarray*}
	 \begin{multline*}
	 	\hspace{2.5cm}\leq F\big(\alpha s_{1}+Es_{1}+\lambda\Psi(\alpha s_{0}+Es_{0},\alpha s_{1}+Es_{1}),\\\alpha t_{1}+Et_{1}+\lambda\Psi(\alpha t_{0}+Et_{0},\alpha t_{1}+Et_{1})\big)
	 \end{multline*}
	 	\begin{multline*}
	\hspace{2.5cm}=F\big((\alpha s_{1},\alpha t_{1})+(Es_{1},Et_{1})+\lambda\{\Psi(\alpha s_{0} +Es_{0},\alpha s_{1}+Es_{1}),\\\Psi(\alpha t_{0}+Et_{0},\alpha t_{1}+Et_{1})\}\big)
	\end{multline*}
		\begin{multline*}
		\hspace{2.5cm}=F\big(\alpha (s_{1},t_{1})+(E\times E)(s_{1},t_{1})+\lambda(\Psi\times\Psi)((\alpha s_{0}+Es_{0},\\\alpha s_{1}+Es_{1}),(\alpha t_{0}+Et_{0},\alpha t_{1}+Et_{1}))\big)
	\end{multline*}
		\begin{multline*}
			\hspace{2.5cm}=F\big(\alpha (s_{1},t_{1})+(E\times E)(s_{1},t_{1})+\lambda(\Psi\times\Psi)(\alpha (s_{0},t_{0})+(E\times E)(s_{0},t_{0}),\\(\alpha (s_{1},t_{1})+(E\times E)(s_{1},t_{1})))\big)
		\end{multline*}
	\begin{eqnarray*}
		\hspace{2.5cm}	&\leq&\max\left\{F\left(Es_{0},Et_{0}\right),F\left(Es_{1},Et_{1}\right)\right\}\\
			&<&\max\left\{g(Es_{0})+\epsilon,g(Es_{1})+\epsilon\right\}\\
				&\leq&\max\{g(Es_{0}),g(Es_{1})\}.\hspace{8.5cm}  
	\end{eqnarray*}
		 
 \noindent
	Therefore, $g(s)=\inf\limits_{t\in S}F(s,t)$ is QSEP  function w.r.t. $\Psi$.
\end{proof} 
\begin{corollary}
	Let $S\subseteq R^{n}$ be a SEI set w.r.t. $\Psi$ and $F\colon \underbrace{S\times S\times\dots\times S}_{n\text{-times}}\rightarrow R$ be a continuous quasi strongly $\underbrace{E\times E\times\dots\times E}_{n\text{-times}}$-preinvex function w.r.t. $\underbrace{\Psi\times \Psi\times\dots\times\Psi}_{n\text{-times}}$, $i.e.,$ $F$ is quasi strongly $\underbrace{E\times E\times\dots\times E}_{n\text{-times}}$-preinvex function to each variable. Then, the function $g\colon S\rightarrow R$ defined by 
	\begin{eqnarray*}
		g(s_{1})= \inf\limits_{s_{2},...,s_{n}\in S}F(s_{1},s_{2},...,s_{n}),
	\end{eqnarray*} is QSEP function w.r.t. $\Psi$.
\end{corollary}
\begin{definition}
	Let $S\subseteq R^{n}$ be a SEI set and $h\colon S\rightarrow R$ be a function. The lower level set of $h$ at $r\in R$ is defined as$\colon$
\begin{eqnarray*}
	K_{r}=\{s\in S\colon h(s)\leq r\}.
\end{eqnarray*}
\end{definition} 

In this theorem, we show an important relationship between lower level sets and QSEP functions.
\begin{theorem}
	Let $S\subseteq R^{n}$ be a SEI set. If the lower level set $K_{r}$ is SEI set $\forall r\in R$, then the function $h\colon S\rightarrow R$ is QSEP on $S$.
\end{theorem}
\begin{proof}
	Assume $S\subseteq R^{n}$ is SEI set and the set $K_{r}$ be SEI for each $r\in R$. For each $s,t\in S,\alpha\in[0,1]$ and $\lambda\in [0,1]$, we have 
	\begin{eqnarray*}
		\alpha t+Et+\lambda\Psi(\alpha s+Es,\alpha t+Et)\in S.
	\end{eqnarray*}
	By using Lemma \ref{lem:1}, we get $Et\in S~ \forall t\in S$.
	Let $r=\max\{h(Es),h(Et)\}$ and $s,t \in K_{r}.$
	Since $K_{r}$ is SEI set, we have 
	\begin{eqnarray*}
		\alpha t+Et+\lambda\Psi(\alpha s+Es,\alpha t+Et)\in K_{r},
	\end{eqnarray*}
	which implies 
	
\begin{eqnarray*}
	h(\alpha t+Et+\lambda\Psi(\alpha s+Es,\alpha t+Et)) \leq r=\max\{h(Es),h(Et)\},
\end{eqnarray*}or
\begin{eqnarray*}
	h(\alpha t+Et+\lambda\Psi(\alpha s+Es,\alpha t+Et))\leq  \max\{h(Es),h(Et)\}.
\end{eqnarray*}
	\noindent
	Therefore, the function $h$ is QSEP on $S$.   
\end{proof}      	

Now, we extend the class of SEI function to quasi strongly $E$-invex (QSEI) function and pseudo strongly $E$-invex (PSEI) function on SEI set as follows$\colon$
\begin{definition} Let $S\subseteq  {R}^{n}$ be a SEI set. A differentiable function $h\colon S\rightarrow  {R}$ is   called QSEI w.r.t. $\Psi$ on $S$, if $\forall s,t\in S$ and $\alpha\in[0,1]$, we have 
\begin{eqnarray*}
	h(Es)\leq h(Et) \implies\nabla h(Et)\Psi(\alpha s+Es,\alpha t+Et)^{T}\leq0.
\end{eqnarray*} 
	For $\alpha=0$, $h$ reduces to $E$-quasiinvex function \cite{Jaiswal}.
\end{definition}

\begin{example}
	Consider $E\colon R^{2}\rightarrow R^{2}$ is defined by $E(s,t)=(0,t)$, and $\Psi\colon R^{2}\times R^{2}\rightarrow R^{2}$ is defined by $\Psi((s_{1},t_{1}),(s_{2},t_{2}))=(s_{1}-s_{2},t_{1}-t_{2})$. The set $S=\{(s,t)\in R^{2}\colon s,t\leq0\}$ is SEI w.r.t. $\Psi$,  and the function $h\colon S\rightarrow R$ is defined by $h(s,t)=s^{3}+t^{3}$. Then, the function $h$ is QSEI w.r.t. $\Psi$ on   $S$. For all $s=(s_{1},t_{1}),t=(s_{2},t_{2})\in S$,~$t_{1}\leq t_{2}$ and $\alpha\in[0,1],$ we have
\begin{eqnarray*}
	h(Es)\leq h(Et) \implies\nabla h(Et)\Psi(\alpha s+Es,\alpha t+Et)^{T}
\end{eqnarray*}
	\begin{eqnarray*}
	\hspace{4.5cm}=3t_{2}^{2}(\alpha+1)(t_{1}-t_{2})\leq0.
		\end{eqnarray*}
	But the function $h$ is not strongly $E$-invex w.r.t. $\Psi$ on $S$. Particularly, at points $t_{1}=-\frac{1}{2},t_{2}=-\frac{1}{4}$, we have
	\begin{eqnarray*}
	\hspace{3.5cm}	0&\leq& h(Es)- h(Et)-\nabla h(Et)\Psi(\alpha s+Es,\alpha t+Et)^{T}\\ &=&t_{1}^{3}-t_{2}^{3}-3t_{2}^{2}(\alpha+1)(t_{1}-t_{2})\\&=&-\frac{7}{64}+\frac{3(\alpha+1)}{64}.
		\end{eqnarray*}
 \end{example}

In this theorem, we show a relation between SEI functions and QSEI functions.
\begin{theorem}
	Let $S\subseteq R^{n}$ be a SEI set. Let $h\colon S\rightarrow R$ be a differentiable SEI function on $S$ and $h(Es)\leq h(Et)$ for every $s,t\in S$. Then, the function $h$ is QSEI on $S$.
\end{theorem}
\begin{proof}
	By the Definition \ref{5P}, the proof is obvious.
\end{proof}

In the following theorem, we show a relationship between SEP and QSEP functions.
\begin{theorem}
	Let $S\subseteq R^{n}$ be a SEI set. Let $h\colon S\rightarrow R$ be a differentiable SEI function on $S$, $\Psi$ satisfies the $\bf {Condition}~{A}$ and  $h(Es)\leq h(Et)$ for every $s,t\in S$. Then, the function $h$ is QSEP on $S$. 
\end{theorem}
\begin{proof}
	Proof is obvious from Theorem \ref{thm:1} and Theorem \ref{thm:6}.
\end{proof}

\begin{definition} Let $S\subseteq  {R}^{n}$ be a SEI set. A differentiable function $h\colon S\rightarrow {R}$ is  called pseudo  strongly $E$-invex (PSEI) w.r.t. $\Psi$ on $S$, if $\forall u,v\in S$ and $\alpha\in[0,1]$, we have 
	\begin{eqnarray*}
		\nabla h(Et)\Psi(\alpha s+Es,\alpha t+Et)^{T}\geq0\implies h(Es)\geq h(Et).
	\end{eqnarray*}  
	
	For $\alpha=0$, $h$ reduces to $E$-pseudo invex function  \cite{Jaiswal}.
\end{definition}

\begin{example}\label{Exm4}
	Consider $E\colon R^{2}\rightarrow R^{2}$ is defined by $E(s,t)=(0,t)$, and $\Psi\colon R^{2}\times R^{2}\rightarrow R^{2}$ is defined by $\Psi((s_{1},t_{1}),(s_{2},t_{2}))=(s_{1}-s_{2},t_{1}-t_{2})$. The set $S=\{(s,t)\in R^{2}\colon s,t\geq0\}$ is SEI w.r.t. $\Psi$, and the function $h\colon S\rightarrow R$ is defined by $h(s,t)=-s^{2}-t^{2}$. Then, the function $h$ is PSEI w.r.t. $\Psi$ on $S$. For all $s=(s_{1},t_{1}),t=(s_{2},t_{2})\in S,~ t_{1}\leq t_{2}$ and $\alpha\in[0,1]$, we have
\begin{eqnarray*}
	\nabla h(Et)\Psi(\alpha s+Es,\alpha t+Et)^{T}\geq0 \implies h(Es)-h(Et)
\end{eqnarray*}
	\begin{eqnarray*}
	\hspace{4cm}=(t_{2}^{2}-t_{1}^{2})\geq0\implies h(Es)\geq h(Et).
		\end{eqnarray*}
	But the function $h$ is not SEI w.r.t. $\Psi$ on $S$. Particularly, at a point $\alpha=0$,
	\begin{eqnarray*}
	\hspace{2cm}0&\leq& h(Es)- h(Et)-\nabla h(Et)\Psi(\alpha s+Es,\alpha t+Et)^{T},\\&=&-(t_{2}-t_{1})^{2}\leq0.
	\end{eqnarray*} 
\end{example}

In this theorem, we show an important relationship between SEI functions and PSEI functions.

\begin{theorem}
	Let $S\subseteq R^{n}$ be a SEI set and $h\colon S\rightarrow R$ be a SEI function on $S$. If 
	\begin{eqnarray*}
		\nabla h(Et)\Psi(\alpha s+Es,\alpha t+Et)^{T}\geq0, ~\forall s,t\in S,\alpha\in[0,1],
		\end{eqnarray*}
	Then, the function $h$ is PSEI on $S$.
\end{theorem}
\begin{proof}
	By the Definition \ref{5P}, the proof is obvious.
\end{proof} 
 
\begin{theorem}\label{thm:11}
	Let $h_{j}\colon {R}^{n}\rightarrow  {R},~~ 1\leq j\leq n$, be QSEP functions w.r.t. $\Psi$  on $ {R}^{n}$. If $E(S)\subseteq S$, then   
	\begin{eqnarray*}
		S=\{s\in  {R}^{n}\colon h_{j}(s)\leq0,~~1\leq j\leq n\}
	\end{eqnarray*}
is SEI set.
\end{theorem}
\begin{proof}
	Since $h_{j}(s),~~ 1\leq j\leq n$, are QSEP functions, for every $s,t\in S\subseteq R^{n},\alpha\in[0,1]~~\&~~\lambda\in[0,1]$, we get
	\begin{eqnarray*}
		h_{j}(\alpha t+Et+\lambda\Psi(\alpha s+Es,\alpha t+Et))&\leq& \max\{(h_{j}(Es) ,(h_{i}(Et)\},\\
		&\leq& 0,
	\end{eqnarray*}
	where we use the assumption $E(S)\subseteq S$ to obtain the right most of the above inequality. 
	Hence, 
\begin{eqnarray*}
	\alpha t+Et+\lambda\Psi(\alpha s+Es,\alpha t+Et)\in S,
\end{eqnarray*}
	which is required results.
\end{proof}
\begin{theorem}\label{thm:12}
	Let $h_{j}\colon {R}^{n}\rightarrow  {R}, ~~1\leq j\leq n$, be a QSEP functions w.r.t. $\Psi$  on $R^{n}$, then  
\begin{eqnarray*}
	S=\bigcap\limits_{j=1}^{n}\left\{s\in R^{n}\colon h_{j}(s)\leq0,~~1\leq j\leq n\right\},
\end{eqnarray*}is a SEI set.
\end{theorem}
\begin{proof}
	By the Theorem \ref{thm:11}, the sets $S_{j}=\{s\in  {R}^{n}\colon h_{j}(s)\leq0,\},~~1\leq j\leq n$, are strongly $E$-invex. Thus, the intersection $\bigcap\limits_{j=1}^{n}S_{j}$ of $S_{j}$ is also SEI set.  
\end{proof}

\section{{Non-linear programming problem} }\label{Sec:Non-linear programming problem}
 
In this section, we consider the $NLPP$ for QSEP functions is known as QSEP programming problem, which generalize the results derived by Iqbal et al. \cite{Hussain2}.

\begin{equation}\label{eqn:1}
	(P)~
	\begin{cases}
		Min~ h_{0}(s)\\
		
		h_{j}(s)\leq 0,~~ j=1,2,...,m,\\
		
		s\in  {R}^{n},
	\end{cases}
\end{equation}

\noindent
 where $h_{j}\colon {R}^{n}\rightarrow  {R},~~0\leq j\leq m$, are QSEP on $ {R}^{n}$.\\ Here, $X$ represents the nonempty set of feasible solutions of QSEP programming problem$\colon$
\begin{equation}\label{eqn:2}
	X=\{s\in  {R}^{n}\colon h_{j}(s)\leq0,~~1\leq j\leq m\}.
\end{equation}

\begin{lemma}\label{lem:3}  
	 If $E(X)\subseteq X$, then $X$ is a SEI set w.r.t. $\Psi$.
\end{lemma}

\begin{proof}
	By the Theorem \ref{thm:11}, the result is obvious.
\end{proof} 
 
\begin{lemma}\label{lemma 4}
	Let  $X$ be a SEI set w.r.t. $\Psi$. If $s\in X$ is feasible solution of (\ref{eqn:1}), then $\alpha s+Es, \forall\alpha\in[0,1]$ is also a feasible solution of (\ref{eqn:1}).
\end{lemma} 
\begin{proof}
	By the Definition \ref{definition}, the proof is obvious.
\end{proof}

\begin{theorem}   
	Let $h_{0}\colon {R}^{n}\rightarrow {R}$ be a strictly QSEP function w.r.t. $\Psi$ on $ {R}^{n}$, $h_{j}\colon {R}^{n}\rightarrow  {R},~ 1\leq j\leq m$, be a QSEP functions w.r.t. $\Psi$ on
	$ {R}^{n}$ and $E(X)\subseteq X$. If $s^{*}$ is a local minimum point of (\ref{eqn:1}), then $s^{*}$ is a strictly minimum point of (\ref{eqn:1}).
\end{theorem}
\begin{proof}
	By Lemma \ref{lemma 4}, for every $s\in X\implies s^{*}= Es\in X$. Since $s^{*}$ is a local minimum point of (\ref{eqn:1}), we get that $s^{*}\in E(X)\subseteq X$, $h_{j}(s^{*})\leq 0,~1\leq j\leq m$, and $\exists~ \epsilon>0$ s.t. 
	\begin{eqnarray*}
		h_{0}(s^{*})\leq h_{0}(t),~ \forall t\in B_{\epsilon}(s^{*})\cap X\setminus{\{s^{*}\}},
	\end{eqnarray*} where $B_{\epsilon}(s^{*})=\{t\in R^{n}\colon\Vert t-s^{*}\Vert<\epsilon\}$. Suppose that $\exists ~~u^{*}\in E(X)\subseteq X,~~u^{*}\neq s^{*}$ s.t. $h_{0}(u^{*})<h_{0}(s^{*}).$~~Since $u^{*},s^{*}\in E(X)\subseteq X$,~~ $ \exists ~~u,s\in X$ s.t. $u^{*}=Eu$, $s^{*}=Es$. For any fixed $\alpha\in[0,1]$ and $\lambda\in[0,1]$, due to Lemma \ref{lem:3}, we get that 
\begin{eqnarray*}
	\alpha s+s^{*}+\lambda\Psi(\alpha u+u^{*},\alpha s+s^{*})\in X.
\end{eqnarray*} 
And 
	\begin{equation}\label{eqn:3}
		h_{0}(\alpha s+s^{*}+\lambda\Psi(\alpha u+u^{*},\alpha s+s^{*}))<\max\{h_{0}(u^{*}),h_{0}(s^{*})\}=h_{0}(s^{*}).
	\end{equation}
	The four possible cases will be arising$\colon$\\
	
	\noindent
	$\mathbf{Case (1)\colon}$ $\Psi(u^{*},s^{*})=0~\&~\alpha=0$, then (\ref{eqn:3}) gives a contradiction.
	
	\vspace{.2cm}
	\noindent
	$\mathbf{Case (2)\colon}$ $\Psi(u^{*},s^{*})\neq 0~\&~\alpha=0$, we choose $\bar{\lambda}=\min\left\{1,\dfrac{\epsilon}{\left\Vert\Psi(u^{*},s^{*})\right\Vert}\right\}$ and, for every $\lambda\in (0,\bar{\lambda})$, we get
	\begin{align*}
		\Vert s^{*}+\lambda \Psi(u^{*},s^{*})-s^{*})\Vert&=\lambda\Vert\Psi(u^{*},s^{*})\Vert\\
		&<\bar{\lambda}\Vert\Psi(u^{*},s^{*})\Vert\\&\leq\epsilon.
	\end{align*} 
	In this case, we obtain that  
	\begin{eqnarray*}
		s^{*}+\lambda\Psi(u^{*},s^{*})\in B_{\epsilon}(s^{*})\cap X\setminus{\{s^{*}\}}
	\end{eqnarray*} 
  $\forall\lambda\in(0,\bar{\lambda})$,
and from (\ref{eqn:3}), we get
 \begin{eqnarray*}
	h_{0}(s^{*} +\lambda\Psi(u^{*},s^{*}))<h_{0}(s^{*}),
\end{eqnarray*}
	which contradicts that $s^{*}$ is a local minimum point of (\ref{eqn:1}).
	
	\vspace{.2cm}
	\noindent
	$\mathbf{Case (3)\colon}$ $\Psi(\alpha u+u^{*},\alpha s+s^{*})=0~\&~\alpha\neq0$, we choose $\bar{\alpha}=\min\left\{1,\dfrac{\epsilon}{\Vert s\Vert}\right\}$ and, for every $\lambda\in (0,\bar{\alpha})$, we get 
	\begin{align*}
		\Vert\alpha s+s^{*}-s^{*})\Vert&=\alpha\Vert s \Vert\\
		&<\bar{\alpha}\Vert s \Vert\\&\leq\epsilon.
	\end{align*}
\noindent	
In this case, we obtain that 
\begin{eqnarray*}
		\alpha s+s^{*} \in B_{\epsilon}(s^{*})\cap X\setminus{\{s^{*}\}}
	\end{eqnarray*}
  $\forall\alpha\in(0,\bar{\alpha})$, 
and from (\ref{eqn:3}), we get
 \begin{eqnarray*}
	h_{0}(\alpha s+s^{*})<h_{0}(s^{*}),
\end{eqnarray*}
	again which contradicts that $s^{*}$ is a local minimum point of (\ref{eqn:1}).
	
	\vspace{.2cm}
	\noindent
	$\mathbf{Case (4)\colon}$ $\Psi(\alpha u+u^{*},\alpha s+s^{*})\neq 0~\&~\alpha\neq0$, we choose $\bar{\alpha}=\min\left\{1,\dfrac{\epsilon_{1}}{\Vert s\Vert}\right\}$ and, for every $\alpha\in (0,\bar{\alpha})$,  we get 
	\begin{align*}
		\left\Vert\alpha s+s^{*}+\lambda \Psi(\alpha u+u^{*},\alpha s+s^{*})-s^{*})\right\Vert&\leq\alpha\Vert s\Vert+\lambda\left\Vert\Psi(\alpha u+u^{*},\alpha s+s^{*})\right\Vert \\
		&<\bar{\alpha}\Vert s\Vert+\lambda\Vert\Psi(\bar{\alpha} u+u^{*},\bar{\alpha} s+s^{*})\Vert\\
		&\leq\epsilon_{1}+\lambda\left\Vert\Psi\left(\dfrac{\epsilon_{1}}{\Vert s\Vert} u+u^{*},\dfrac{\epsilon_{1}}{\Vert s \Vert} s+s^{*}\right)\right\Vert,
	\end{align*}
	\noindent
	again  we choose $\bar{\lambda}=\min\left\{1,\dfrac{\epsilon_{2}}{\left\Vert\Psi\left(\dfrac{\epsilon_{1}}{\left\Vert s\right\Vert} u+u^{*},\dfrac{\epsilon_{1}}{\left\Vert s\right\Vert} s+s^{*}\right)\right\Vert}\right\}$ and, for any $\lambda\in (0,\bar{\lambda})$,
	\begin{align*}
		\hspace{7.3cm}	&<\epsilon_{1}+\bar{\lambda}\left\Vert\Psi\left(\dfrac{\epsilon_{1}}{\Vert s\Vert} u+u^{*},\dfrac{\epsilon_{1}}{\Vert s\Vert} s+s^{*}\right)\right\Vert.\\
		&\leq\epsilon_{1}+\epsilon_{2}\\&=\epsilon.	 
	\end{align*}
	\noindent
In this case, we obtain that 
\begin{eqnarray*}
		\alpha s+s^{*}+\lambda\Psi(\alpha u+u^{*},\alpha s+s^{*})\in B_{\epsilon}(s^{*})\cap X\setminus{\{s^{*}\}}
	\end{eqnarray*}
  $\forall\alpha\in (0,\bar{\alpha})$, $\lambda\in(0,\bar{\lambda})$, and from (\ref{eqn:3}), we get
   \begin{eqnarray*}
	h_{0}(\alpha s+s^{*}+\lambda\Psi(\alpha u+u^{*},\alpha s+s^{*}))<h_{0}(s^{*}),
\end{eqnarray*}
	again which contradicts that $s^{*}$ is a local minimum point of (\ref{eqn:1}). 
\end{proof}
\vspace{.2cm}

\section{\bf Conclusion}\label{Concl5}
The notion of quasi strongly $E$-preinvexity has been introduced, and sufficient non-trivial examples have been established in support of our definitions. The vital relationships of these functions have been discussed and several interesting properties have been explored. An application to a NLPP for QSEP functions has also been considered.  Our results extend the previously known results done by many researchers; see \cite{ Azagra,Poury1,Hussain2,Fulga,Jaiswal}. This work can be explored over Riemannian manifolds in the future.

\vspace{.2cm}
All the above defined concepts are useful for solving  problems in real life, such as mathematical programming, optimization problems, variational inequalities, and equilibrium problems; for more details, see \cite{Hussain2,Jaiswal,Jeyakumar,Noor,Yang}.

\end{document}